\documentclass[11pt]{article}

\RequirePackage[l2tabu, orthodox]{nag}
\usepackage[T1]{fontenc}
\usepackage[utf8]{inputenc}
\usepackage{lmodern}
\usepackage[english]{babel}
\usepackage{color} 
\usepackage[usenames,dvipsnames]{xcolor}
\usepackage{amsrefs,amsthm,amssymb,amsmath}
\usepackage[pdftex]{graphicx}
\usepackage{enumerate}
\usepackage{url}

\numberwithin{equation}{section}
\numberwithin{figure}{section}

\def\comment#1{}

\newcommand{\R}{\mathbb{R}}
\newcommand{\N}{\mathbb{N}}
\newcommand{\D}{\mathbb{D}}
\newcommand{\Z}{\mathbb{Z}}

\newcommand{\T}{\mathbb{S}^1}

\newcommand{\Homeo}{\mathsf{Homeo}}

\newcommand{\stab}[2]{\mathsf{Stab}_{#1}(#2)}

\newcommand{\PSL}{\mathsf{PSL}}
\newcommand{\SL}{\mathsf{SL}}

\newcommand{\mP}{\mathsf{P}}

\newcommand{\Fix}{\mathsf{Fix}}
\newcommand{\Aut}{\mathsf{Aut}}

\newcommand{\ord}{\mathsf{ord}}

\DeclareMathOperator{\len}{\mathrm{length}}
\DeclareMathOperator{\CO}{\mathsf{CO}}
\DeclareMathOperator{\LO}{\mathsf{LO}}

\newtheorem{thm}{Theorem}[section]
\newtheorem*{thm*}{Theorem}
\newtheorem{claim}{Claim}

\newtheorem{lem}[thm]{Lemma}
\newtheorem{prop}[thm]{Proposition}
\newtheorem{cor}[thm]{Corollary}
\theoremstyle{definition}
\newtheorem{dfn}[thm]{Definition}
\newtheorem{con}[thm]{Convention}
\theoremstyle{remark}
\newtheorem{rem}[thm]{Remark}
\newtheorem{ex}[thm]{Example}

\usepackage[margin=2.5cm]{geometry}

\usepackage{indentfirst}
\usepackage{microtype}

\usepackage[colorlinks=true,linkcolor=blue,citecolor=magenta]{hyperref}

\title{Ping-pong configurations and circular orders on free groups}
\date{}
\author{
Dominique Malicet
\thanks{LAMA,  Universit\'e  Paris-Est  Marne-la-Vall\'ee,  CNRS UMR 8050, 5  bd.~Descartes,  77454  Champs  sur Marne, France.} 
\and 
Kathryn Mann
\thanks{Dept.~Mathematics, Brown University, 151 Thayer Street,
Providence, RI 02912.}  
\and 
Crist\'obal Rivas
\thanks{Dpto.~Matem\'atica y C.C.~Universidad de Santiago de Chile,
Alameda 3363, Estaci\'on Central, Santiago, Chile.}
\and  
Michele Triestino
\thanks{IMB, Universit\'e Bourgogne Franche-Comt\'e, CNRS UMR 5584,
9 av.~Alain Savary,
21000 Dijon, France.}}
\begin{document}
\maketitle

\begin{abstract}
We discuss actions of free groups on the circle with ``ping-pong'' dynamics; these are dynamics determined by a finite amount of combinatorial data, analogous to Schottky domains or Markov partitions.   Using this, we show that the free group $F_n$ admits an isolated circular order if and only if $n$ is even, in stark contrast with the case for linear orders.  This answers a question from \cite{MR}. Inspired by work in \cite{markov-partition-vfree}, we also exhibit examples of ``exotic'' isolated points in the space of all circular orders on $F_2$.   Analogous results are obtained for linear orders on the groups $F_n\times \mathbb Z$.\footnote{\textbf{MSC\textup{2010}:} Primary 20F60, 57M60.
Secondary  20E05, 37C85, 37E05, 37E10, 57M60.
}
\end{abstract}

\section{Introduction}

Let $G$ be a group. A (left-invariant) \emph{linear order}, often called a \emph{left order} on $G$ is a total order invariant under left multiplication.  Left-invariance directly implies that the order is determined by the set of elements greater than the identity, called the \emph{positive cone}.  It is often far from obvious whether a given order can be determined by only finitely many inequalities, or whether a given group admits such a finitely-determined order.   This latter question turns out to be quite natural from an algebraic perspective, and can be traced back to Arora and McCleary \cite{Arora-McCleary} for the special case of free groups.  McCleary answered this question for free groups shortly afterwards, showing that $F_n$ has no finitely determined orders~\cite{McCleary2}. 

The question of finite determination gained a topological interpretation following Sikora's definition of the \emph{space of linear orders} on $G$ in \cite{sikora}.  This space, denoted $\LO(G)$, is the set of all linear orders on $G$ endowed with the topology generated by open sets 
\[U_{(\preceq, X)} :=\{ {\preceq'} \mid {x\preceq' y } \text{ iff } {x \preceq y} \text{ for all } x,  y \in X\}\]
as $X$ ranges over all finite sets of $G$.   Finitely determined linear orders on $G$ are precisely the \emph{isolated points} of $\LO(G)$;  going forward, we will refer to these as \emph{isolated orders}.  
This correspondence between isolated points and finitely determined orders is perhaps the simplest instance of the general theme that topological properties of $\LO(G)$ should reflect algebraic properties of $G$.   

Presently, several families of groups are known to either admit or fail to admit isolated orders, with proofs that use both purely algebraic and dynamical methods.  Some examples of groups that do \emph{not} admit isolated orders include free abelian groups \cite{sikora}, free groups \cite{McCleary2, NavasFourier}, free products of arbitrary linearly orderable groups \cite{RivasLO}, and some amalgamated free products such as fundamental groups of orientable closed surfaces \cite{ABR}.    Large families of groups which do have isolated orders include braid groups \cite{DD,DehornoyBook}, groups of the form $\langle x,y\mid x^n=y^m\rangle$ ($n, m\in \Z$) \cite{ItoAlg,NavasAlg}, and groups with triangular presentations \cite{dehornoy14}.  (In fact, all of these latter examples have orders for which the positive cone is finitely generated as a semi-group, a strictly stronger condition.)
As a consequence of our work here, we give a family of groups where, interestingly, \emph{both} behaviors occur.  

\begin{thm} \label{LO_thm}
Let $F_n$ denote the free group of $n$ generators.  The group $F_n\times \Z$ has isolated linear orders if and only if $n$ is even.
\end{thm}

This result appears to give the first examples of any group $G$ with a finite index subgroup $H$ (in this case $F_n \times \Z \subset F_m \times \Z$, for $n$ odd and $m$ even) such that $\LO(G)$ and $\LO(H)$ are both infinite, but only $G$ contains isolated points.  

Theorem \ref{LO_thm} also has an interesting consequence regarding the \emph{space of marked groups}.   
As shown in \cite[Prop.~2.13]{champ-guira}, the set of left-orderable groups is a closed subset of the space of marked groups on $n$ generators.  However, Theorem \ref{LO_thm} implies that this is not the case either for the subset of groups admitting isolated linear orders (or its complement): one may take a sequence of markings of  $F_2 \times \Z$ so as to approach $F_3\times \Z$, and similarly, a sequence of markings on $F_3\times\Z$ can be chosen to approach $F_4\times \Z$ (see \cite[\S 2.4]{champ-guira}).  Thus, Theorem \ref{LO_thm} immediately gives the following. 

\begin{cor}\label{cor:marked}
In the space of finitely generated marked groups, having an isolated linear order is neither a closed, nor an open property.
\end{cor}

The main tool for Theorem \ref{LO_thm}, and main focus of this work, is the study of \emph{circular orders} on $F_n$ and the dynamics of their corresponding actions of $F_n$ on $\T$.   It is well known that, for countable $G$, admitting a linear order is equivalent to acting faithfully by orientation-preserving homeomorphisms on the line.  In the same vein, a \emph{circular order} on $G$ is an algebraic condition which, for countable groups, is equivalent to acting faithfully by orientation-preserving homeomorphisms on $\T$.  We recall the definition and basic properties in Section \ref{CO_sec}.  
Any action of $G$ on $\T$ lifts to an action of a central extension of $G$ by $\Z$ on the line, giving us a way to pass between circular and linear orders on these groups, and giving us many dynamical tools for their study.  

Analogous to $\LO(G)$, one can define a \emph{space of circular orders} $\CO(G)$.  In \cite{MR}, the second and third authors showed that a circular order on $F_n$ is isolated if and only if the corresponding action on the circle has what they called \emph{ping-pong dynamics}.  They gave examples of isolated circular orders on free groups of even rank, but the odd rank case was left as an open problem.  Here we answer this question in the negative: 

\begin{thm}  \label{CO_thm}
$F_n$ admits an isolated circular order if and only if $n$ is even.  
\end{thm}

Similarly to Corollary~\ref{cor:marked}, one can also prove that the set of groups admitting isolated circular orders is neither closed nor open in the space of marked groups.

We prove Theorem \ref{CO_thm} by developing a combinatorial tool for the study of actions on $\T$ with ping-pong dynamics (similar to actions admitting Markov partitions), inspired by the work in \cite{DKN2014} and \cite{markov-partition-vfree}.  We expect these to have applications beyond the study of linear and circular orders; one such statement is given in Theorem~\ref{t:gen_ping-pong}.  
The notion of ping-pong dynamics is defined and motivated in the next section.  Sections \ref{CO_sec} and \ref{LO_sec} give the application to the study of circular and linear orders, respectively, and the proofs of Theorems \ref{LO_thm} and \ref{CO_thm}.

\section{Ping-pong actions and configurations}
\label{s:PingPong}

\begin{dfn} \label{pp_action_def}
Let $G=F_n$ be the free group of rank $n$, freely generated by $S=\{ a_1,\ldots,a_n\}$.   A \emph{ping-pong action} of $(G, S)$ on $\T$ is a representation $\rho:G\to \Homeo_+(\T)$ such that there exist pairwise disjoint open sets $D(a)\subset \T$, $a\in S\cup S^{-1}$, each of which has finitely many connected components, and such that $\rho(a) \left (\T \setminus D(a^{-1})\right ) \subset \overline{D(a)}$.    We further assume that if $I$ and $J$ are any connected components of $D(a)$, then $\bar{I} \cap \bar{J} = \emptyset$.

We call the sets $D(a)$ the \emph{ping-pong domains} for $\rho$.  
\end{dfn}
A similar definition is given in \cite{MR}, with the additional requirement that ping-pong domains be closed.  The above, more general definition is more natural for our purposes, although we will later introduce Convention~\ref{Closures} to reconcile the two.  
The reader may notice that, for a given ping-pong action $\rho$ of $(G, S)$, there can be many choices of sets $D(a)$ satisfying the property in Definition \ref{pp_action_def}.    For instance, if $\rho$ is a ping-pong action such that $\bigcup_{a \in S \cup S^{-1}} \overline{D(a)} \neq \T$, then one may choose an arbitrary open set $I$ disjoint from $\bigcup_{a \in S \cup S^{-1}} \overline{D(a)}$ and replace $D(a_1)$ with $D(a_1) \cup I$, leaving the other domains unchanged.  These new domains still satisfy $\rho(a) \left (\T \setminus D(a^{-1})\right ) \subset \overline{D(a)}$.     Later we will adopt a convention to avoid this kind of ambiguity.

\paragraph{Motivation: why ping-pong actions?} The classical ping-pong lemma implies that ping-pong actions are always faithful, and a little more work shows that the action is determined up to semi-conjugacy by a finite amount of combinatorial data coming from the cyclic ordering and the images of the connected components of the sets $D(a)$ (see Definition \ref{pp_config_def} and Lemma \ref{config_determines} below, or \cite[Thm.~4.7]{BP}).  
In particular, one can think of ping-pong actions as the family of ``simplest possible'' faithful actions of $F_n$ on $\T$, and it is very easy to produce a diverse array of examples.    Perhaps the best-known examples are the actions of discrete, free subgroups of $\PSL(2,\R)$  on $\R \mP^1$.  For these actions, one can choose domains $D(a)$ with a single connected component.  Figure \ref{PP-0_fig} shows an example of the dynamics of such an action of $F_2 = \langle a, b \rangle$.   

\begin{figure}
\[
\includegraphics[scale=1]{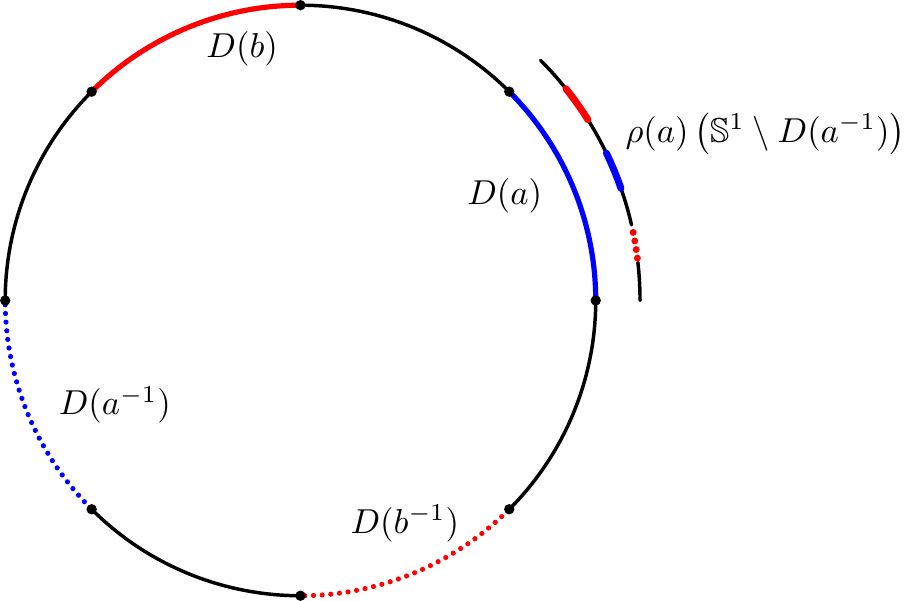}
\]
\caption{Classical ping-pong on two generators}\label{PP-0_fig}
\end{figure}

Despite their simplicity, ping-pong actions are quite useful.  For instance, in \cite{markov-partition-vfree} ping-pong actions were used to construct the first known examples of discrete groups of real-analytic circle diffeomorphisms acting minimally, but not conjugate to a subgroup of a finite central extension of $\PSL(2,\R)$. 
This was a by-product of a series of papers concerning longstanding open conjectures of Hector, Ghys and Sullivan on the relationship between minimality and ergodicity of a codimension-one foliation (see for instance \cite{DKN2009,FK2012,DKN2014}).   In general, it is quite tractable to study the dynamic and ergodic properties of a ping-pong action (or a Markov system), and this program has been carried out by many authors \cite{CC1,CC2,IM,Matsumoto,mane}.

\subsection{Basic properties}

\begin{lem}  \label{reduced_lem}
Given a ping-pong action of $(G, S)$, there exists a choice of ping-pong domains $D(a)$ such that $\rho(a)\left (\T \setminus \overline{D(a^{-1})}\right ) = D(a)$ holds for all $a \in S \cup S^{-1}$.  
\end{lem}

\begin{proof}
Let $\rho$ be a ping-pong action with sets $D(a)$ given.  We will modify these domains to satisfy the requirements of the lemma.  For each generator $a \in S$ (recall this is the free, \emph{not} the symmetric, generating set), we shrink the domain $D(a)$, setting 
$D'(a) :=  \rho(a)\left (\T \setminus \overline{D(a^{-1})}\right ).$
Applying $a^{-1}$ to both sides of the above expression gives $\rho(a^{-1})(D'(a)) = \T \setminus \overline{D(a^{-1})}$. 
Moreover, since the connected components of $D(a)$ have disjoint closures, and the same holds for $D(a^{-1})$, hence also for $D'(a)$, we also have $\rho(a^{-1})\left (\overline{D'(a)}\right ) = \T \setminus D(a^{-1})$; or equivalently 
$\rho(a^{-1})\left (\T \setminus \overline{D'(a)}\right ) = D(a^{-1})$.
This is what we needed to show.
\end{proof} 

\begin{con}  \label{reduced_con}
From now on, we assume all choices of domains $D(a)$ for every ping-pong action are as in Lemma \ref{reduced_lem}.  In particular, this means that, for each $a \in S$, the sets of connected components $\pi_0(D(a))$ and $\pi_0(D(a^{-1}))$ have the same cardinality, and $\rho(a)$ induces a bijection between the connected components of $\T \setminus  \overline{D(a^{-1})}$ and connected components of $D(a)$.  
\end{con}

\begin{dfn} \label{pp_config_def}
Let $\rho$ be a  ping-pong action of $(G, S)$.  The \emph{ping-pong configuration} of $\rho$ is the data consisting of
\begin{enumerate}
\item the cyclic order of the connected components of $\bigcup_{a\in S\cup S^{-1}}D(a)$ in $\T$, and 
\item for each $a\in S\cup S^{-1}$, the assignment of connected components
\[
\lambda_a:\pi_0\left (\bigcup_{b\in S\cup S^{-1}\setminus \{a^{-1}\}}D(b)\right )\to \pi_0\left (D(a)\right )
\]
induced by the action.  
\end{enumerate}
\end{dfn}

Note that {\em not} every abstract assignment $\lambda_a$ as in the definition above can be realized by an action $F_2 \to \Homeo_+(\T)$.   The following construction gives one way to produce some large families of examples. 

\begin{ex}[An easy construction of ping-pong actions] \label{easy_ex}
For $a \in S$, let $X_a$ and $Y_a \subset \T$ be disjoint sets each of cardinality $k(a)$, for some integer $k(a) \geq 1$, such that every two points of $X_a$ are separated by exactly one point of $Y_a$.  Choose these so  that all the sets $X_a \cup Y_a$ are pairwise disjoint as $a$ ranges over $S$.   Let $D(a)$ and $D(a^{-1})$ be neighborhoods of $X_a$ and $Y_a$, respectively, chosen small enough so that all these sets remain pairwise disjoint.   Now one can easily construct a piecewise linear homeomorphism (or even a smooth diffeomorphism) $\rho(a)$ with $X_a$ as its set of attracting periodic points, and $Y_a$ as the set of repelling periodic points such that $\rho(a)\left (\T \setminus D(a^{-1})\right ) = D(a)$.     The assignments $\lambda_a$ are now dictated by the period of $\rho(a)$ and the cyclic order of the sets $X_a$ and $Y_a$.  
\end{ex}

While the reader should keep the construction above in mind as a source of examples, we will show in Example \ref{ex:strange} that not every ping-pong configuration can be obtained in this manner.
However, the regularity (PL or smooth) in the construction is attainable in general.   The following construction gives one possibility for a PL realization that will be useful later in the text.  We leave the modifications for the smooth case as an easy exercise.

\begin{lem}\label{GoodRealization}
Given a ping-pong action $\rho_0$ of $(G,S)$ with domains $\{D_0(a)\}_{a\in S\cup S^{-1}}$ following Convention \ref{reduced_con}, one can find another ping-pong action $\rho$ of $(G,S)$ with domains $\{D(a)\}_{a\in S\cup S^{-1}}$such that:
\begin{enumerate}
\item the action $\rho$ is piecewise linear,  
\item there exists $\mu>1$ such that for any $a\in S\cup S^{- 1}$, one has $\rho(a)'\vert_{D(a^{-1})}\ge \mu$, and 
\item the actions $\rho_0$ and $\rho$ have the same ping-pong configuration.
\end{enumerate}
\end{lem}

\begin{proof}
Let $\rho_0$ be as in the statement of the Lemma.  For each $a \in S \cup S^{-1}$, replace the original domains by smaller domains $D(a) \subset D_0(a)$, chosen small enough so that the largest connected component of $D(a)$ is at most half the length of the smallest connected component of $\T \setminus D(a^{-1})$.  We require also that $D(a)$ has exactly one connected component in each connected component of $D_0(a)$.  
Now define $\rho(a)$ as a piecewise linear homeomorphism that maps connected components of $\T\setminus D(a^{-1})$ onto connected components of $D(a)$ linearly following the assignment $\lambda_a$.  
\end{proof}

The next definition and proposition give a further means of encoding the combinatorial data of a ping-pong action.  This will be used later in the proof of Theorem \ref{CO_thm}.

\begin{dfn}\label{d:graph}
Let $\rho$ be a ping-pong action of $(G, S)$ with domains $D(a)$. 
For each $a\in S$, we define an oriented bipartite graph $\mathsf{\Gamma}_a$ with vertex set equal to $\pi_0(D(a)) \cup \pi_0(D(a^{-1}))$, and edges defined as follows:
\begin{itemize}
\item For $I^+\in \pi_0(D(a))$, let $J^+$ denote the connected component of $\T\setminus D(a)$ adjacent to $I^+$ on the right.  Put an oriented edge from $I^+$ to an interval $I^- \in \pi_0(D(a^{-1}))$ if and only if $\rho(a^{-1})(J^+)= \overline{I^-}$. 
\item Similarly, for $I^- \in \pi_0(D(a))$, with $J^-$ the adjacent interval of  $\T\setminus D(a^{-1})$ on the right, put an oriented edge from $I^-\in \pi_0(D(a^{-1}))$ to $I^+\in \pi_0(D(a))$ if and only if $\rho(a)(J^-)= \overline{I^+}$.
\end{itemize}
\end{dfn}

\begin{prop}\label{p:graph}
Let $\rho:G\to \Homeo_+(\T)$ be a ping-pong action of a free group $(G,S)$.  Then, for each generator $a\in S$, there exists $k(a) \in \N$ such that the graph $\mathsf\Gamma_a$ is an oriented $2k(a)$-cycle.
\end{prop}

\begin{proof}
First, the construction of the graph ensures that it is bipartite, and that each vertex has at most one outgoing edge. 
As a consequence of Convention~\ref{reduced_con}, for any $s\in S\cup S^{-1}$ and connected component $J$ of $\T\setminus D(s^{-1})$, there exists $I\in \pi_0(D(s))$ such that $\rho(s)(J)= \overline{I}$, so each vertex does indeed have an outgoing edge.  
Moreover, if $J' \neq J$ is a different connected component, then $\rho(s)(J) \cap \rho(s)(J') = \emptyset$, so each vertex $I$ has a unique incoming edge.   This shows that $\mathsf \Gamma_a$ is a union of disjoint cycles, and it remains only to prove that the graph is connected.  

To show connectivity, 
let $I^{-}$ be a connected component of $D(a^{-1})$ and consider 
the connected component $J^+$ of $\T\setminus D(a)$ such that $\rho(a^{-1})(J^+)=\overline{I^-}$.
Let $I^+_1$ and $I^+_2$ be the connected components of $D(a)$ (possibly the same) which are adjacent to $J^+$ on either side.  
By definition of the graph $\mathsf{\Gamma}_a$, the intervals $I^+_1,I^-,I^+_2$ are consecutive vertices in the same cycle of the graph. And vice versa: if three intervals  $I^+_1,I^-,I^+_2$ are consecutive vertices, then $J^+:=\rho(a)(\overline{I^-})$ is the connected component of $\T\setminus D(a)$ adjacent to both $I^+_1$ and $I^+_2$.

This proves that if $I^+_1$ and $I^+_2$ are consecutive connected components of $D(a)$ in $\T$, then they belong to the same cycle in $\mathsf\Gamma_a$. Hence we easily deduce that \emph{all} connected components of $D(a)$ are in the same cycle in $\mathsf\Gamma_a$.  The same also holds for the components of $D(a^{-1})$, and the graph is connected.  
\end{proof} 


\section{Left-invariant circular orders}  \label{CO_sec}

We begin by quickly recalling standard definitions and properties.  A reader familiar with circular orders may skip to Section \ref{dynam_sec}.  

\begin{dfn}\label{dfn:CO}
Let $G$ be a group. A \emph{left-invariant circular order} is a function
$c:G\times G\times G\to \{0,\pm 1\}$
such that
\begin{enumerate}
\item $c$ is \emph{homogeneous}: $c(\gamma g_0,\gamma g_1,\gamma g_2)=c(g_0,g_1,g_2)$ for any $\gamma,g_0,g_1,g_2\in G$;
\item $c$ is a \emph{$2$-cocycle} on $G$:
\[
c(g_1,g_2,g_3)-c(g_0,g_2,g_3)+c(g_0,g_1,g_3)-c(g_0,g_1,g_2)=0\quad\text{for any }g_0,g_1,g_2,g_3\in G;
\]
\item\label{cond:3} $c$ is \emph{non-degenerate}: $c(g_0,g_1,g_2)=0$ if and only if $g_i=g_j$ for some $i\neq j$.
\end{enumerate}
The \emph{space} of all left-invariant circular orders on $G$, denoted $\CO(G)$, is the set of all such functions,
endowed with the subset topology from $\{0,\pm 1\}^{G\times G\times G}$ (with the natural product topology).
\end{dfn}

Although spaces of left-invariant linear orders have been well-studied, there are very few cases where the topology of $\CO(G)$ is completely understood.   Other than a few sporadic examples, the only complete description of spaces of circular orders known to the authors comes from \cite{CMR}, which gives a classification of all groups such that $\CO(G)$ is finite, and also a proof that $\CO(A)$ is homeomorphic to a Cantor set for any Abelian group $A$.
 
Given that left-orders on free groups are well understood, a natural next case of circular orders to study is $\CO(F_n)$.  
Our main tool for this purpose is the following classical relationship between circular orders and actions on~$\T$ (see~\cite{Calegari, MR}).  

\begin{prop}\label{p:realization}
Given a left-invariant circular order $c$ on a countable group $G$, there is an action $\rho_c:G\to \Homeo_+(\T)$ such that $c(g_0,g_1,g_2) = \ord\left (\rho_c(g_0)(x), \rho_c(g_1)(x), \rho_c(g_2)(x)\right )$ for some $x \in \T$, where $\ord$ denotes cyclic orientation.   

Moreover, there is a canonical procedure for producing $\rho_c$ which gives a well-defined conjugacy class of action.  This conjugacy class is called the \emph{dynamical realization of $c$ with basepoint $x$}. 
\end{prop}

A description of this procedure is given in \cite{MR}, modeled on the analogous linear case (see e.g.~\cite{Ghys}).  Note that modifying a dynamical realization by \emph{blowing up} the orbit of some point $y \notin \rho(G)(x)$ may result in a non-conjugate action that still satisfies the property $c(g_0,g_1,g_2) = \ord\left (\rho_c(g_0)(x), \rho_c(g_1)(x), \rho_c(g_2)(x)\right )$. However, this non-conjugate action cannot be obtained through the canonical procedure.

\begin{rem}\label{r:induce}
The converse to the above proposition is also true:  if $G$ is a countable subgroup of $\Homeo_+(\T)$, then $G$ admits a circular order.    A proof is given in \cite[Thm.~2.2.14]{Calegari}\footnote{In \cite[Prop.~2.4]{BS} the authors propose an alternative way of inducing an ordering of $G$, different from that in \cite{Calegari}. However their method is incorrect, as the following example shows: suppose to have three distinct homeomorphisms $f,g,h$, with $f$ coinciding with $g$ on one half circle and with $h$ on the other half. Then for any point $x\in \T$, there are always two equal points in the triple $(f(x),g(x),h(x))$.}.  
As a special case, if $\rho:G \to \Homeo_+(\T)$ is such that some point $x$ has trivial stabilizer, then we may define an \emph{induced order} on $G$ by 
\[c(g_1, g_2, g_3):=\ord(\rho(g_1)(x),\rho(g_2)(x),\rho(g_3)(x)).
\]
\end{rem}

While one cannot expect in general to find a point with trivial stabilizer, this does hold for ping-pong actions by the following lemma. 

\begin{lem}   \label{config_determines}
Suppose that $\rho$ is a ping-pong action of $(G, S)$ with domains $D(a)$.   If $x_0 \in \T \setminus \bigcup_{a \in S \cup S^{-1}} D(a)$, then the orbit of $x_0$ is free and its cyclic order is completely determined by the cyclic order of the elements of $\left \{ \pi_0 \left( \bigcup_{a\in S\cup S^{-1}}D(a) \right), \{x_0\} \right \}$ and the assignments $\lambda_a$.  
\end{lem}

The proof is obtained by a careful reading of the standard proof of the classical ping-pong lemma. Details are given in \cite[Lemma 4.2]{MR}.


\subsection{Isolated circular orders on free groups}
\label{dynam_sec}

In this section we will use ping-pong actions to prove Theorem \ref{CO_thm} from the introduction.   As this builds on the framework of \cite{MR}, we start by introducing two results obtained there.  

Let $G$ be any group, and $\rho: G \to \Homeo_+(\T)$.  Recall that, if $\rho(G)$ does not have a finite orbit, then there is a unique closed, $\rho(G)$-invariant set contained in the closure of every orbit, called the \emph{minimal set} of $\rho(G)$.  We denote this set by $\Lambda(\rho)$.  If $\Lambda(\rho) = \T$, the action is called \emph{minimal}.  Otherwise, $\Lambda(\rho)$ is homeomorphic to a Cantor set and $\rho$ permutes the connected components of $\T \setminus \Lambda(\rho)$.  While, for many examples of actions, the permutation will have many disjoint cycles, the next lemma states that this is not the case for dynamical realizations.

\begin{lem}[\cite{MR}  Lemma 3.21 and Cor.~3.24] \label{transitive_lem}
Let $\rho:G\to \Homeo_+(\T)$ be a dynamical realization of a circular order $c$. Suppose that $\rho$ has a minimal invariant Cantor set $\Lambda(\rho)$. Then $\rho$ acts transitively on the set of connected components of  $\T\setminus \Lambda(\rho)$.  
\end{lem}

Since a ping-pong action of a free group of rank at least 2 cannot have finite orbits, invariance of the minimal set immediately implies that $\Lambda(\rho) \subset \bigcup_{a \in S \cup S^{-1}} \overline{D(a)}$.   If additionally, for each $s \neq t \in S \cup S^{-1}$, one has $\overline{D(s)} \cap \overline{D(t)} = \emptyset$, then invariance of $\Lambda(\rho)$ and the definition of ping-pong implies that in fact $\Lambda(\rho) \subset  \bigcup_{a \in S \cup S^{-1}} {D(a)}$.  
Going forward, it will be convenient to have the this stronger condition, which is given by following lemma.  

\begin{lem}
Let $\rho_0$ be a ping-pong action of $(G, S)$ with domains $D_0(a)$ for $a \in S \cup S^{-1}$.  Then there exists an action $\rho$ with the same ping-pong configuration as $\rho_0$ and with domains $D(a)$ satisfying $\overline{D(s)} \cap \overline{D(t)} = \emptyset$ whenever $s \neq t$.    
\end{lem}

\begin{proof}
Let $\rho_0$ be a ping-pong action.   There are only finitely many points $x$ contained in sets of the form $\overline{D_0(s)} \cap \overline{D_0(t)}$ for $t \neq s \in S \cup S^{-1}$.  For each such point $x$, blow up its orbit, replacing each point $y\in \rho_0(G)(x)$ with an interval $I_{y}$; if lengths of the $I_{y}$ are chosen so their sum converges, then we obtain a new circle, say $ \hat{\mathbb S}^1$, with a natural continuous, degree one map $h:  \hat{\mathbb S}^1 \to \T$ given by collapsing each $I_{y}$ to the point $y$.   For each $s\in S\cup S^{-1}$, let $D_1(s) \subset \hat{\mathbb S}^1$ be the preimage of $D_0(S)$ under $h$.  

  Since $x$ has trivial stabilizer and $G$ is free, we may extend the action of $G$ to this new circle by allowing $a \in S$ to act as any orientation-preserving map from $I_y$ to $I_{a(y)}$.  We now show that we may choose maps in such a way as to achieve a ping-pong action with the desired properties.   

For each inserted interval $I = [p,q]$ that is adjacent to a set of the form $D_1(s)$ on the left and $D_1(t)$ on the right (where $s, t \in S \cup S^{-1}$), fix points $p_s < p_t$ in the interior of $I$ and extend $D_1(s)$ into $I$ to include $[p, p_s)$ and $D_1(t)$ to include $(p_t, q]$.   Having done this on each such interval, let $D(s)$ denote the new extended domains, and note that these have disjoint closures.  Now for $a \in S$, define the action of $a$ on such an interval $I = I_y$ as follows.  If $\rho_0(a)(y) \in D_0(a)$, the restriction of $\rho$ to $I$ may be any orientation-preserving homomorphism between $I_y$ and $I_{a(y)}$.   Otherwise, $I_y$ is adjacent to $D_1(a^{-1})$ either on the right or the left, and we define $\rho(a)$ on $I_y$ to map the chosen point $p_{a^{-1}} \in I_y$ to the point $p_a$ in $I_{a(y)}$.  This ensures that $\rho(a)(S^1 \setminus \overline{D(a^{-1})}) = D(a)$, so that these are indeed ping-pong domains for the action.  
Finally, note that by construction, the ping-pong configuration has not changed.  
\end{proof}

\begin{con}\label{Closures}
In a ping-pong action of $(G, S)$, we assume from now on that the domains $D(s)$ satisfy $\overline{D(s)}\cap \overline{D(t)}=\emptyset$  whenever $s \neq t\in S\cup S^{-1}$.  
\end{con}

It follows easily from invariance of $\Lambda(\rho)$ and the definition of ping-pong that, for actions as in Convention \ref{Closures}, we have the inclusion $\Lambda(\rho)\subseteq \bigcup_{a \in S \cup S^{-1}} {D(a)}$.
The following theorem from \cite{MR} relates circular orders and ping-pong actions.   

\begin{thm}[\cite{MR} Thm.~1.5]  \label{alternate_thm}
Let $G=F_n$ be a free group. A circular order $c\in \CO(G)$ is isolated if and only if its dynamical realization $\rho_c:G\to \Homeo_+(\T)$ is a ping-pong action satisfying Convention~\ref{Closures}.
\end{thm} 

With these tools, we proceed to the main goal of this section.  

\begin{proof}[Proof of Theorem \ref{CO_thm}]
The case where $n$ is even is covered in \cite{MR}.  As explained there, the representation of $G$ into $\PSL(2,\R)$ coming from a hyperbolic structure on a genus $n/2$ surface with one boundary component gives an isolated circular order.  (In fact, by taking lifts to cyclic covers, one can obtain infinitely many isolated circular orders in distinct equivalence classes under the action of $\Aut(F_n)$ on $\CO(G)$.)  

To show that $F_n$ does not admit an isolated circular order when $n$ is odd, we need more work.   We begin with some generalities, applicable to free groups of any rank (even or odd).  
Suppose that $\rho: F_n\to \Homeo_+(\T)$ is a dynamical realization of an isolated circular order,  
and fix a free generating set $S = \{a_1, \ldots, a_n\}$ for $F_n$.   By Theorem~\ref{alternate_thm} and  Lemma~\ref{transitive_lem}, $\rho$ is a ping-pong action with domains satisfying Convention \ref{Closures}  and the connected components of $\T\setminus \Lambda(\rho)$ form a unique orbit. Let $c_0,\ldots, c_r$ be the (finitely many) connected components of $\T\setminus \Lambda(\rho)$ that are not contained in any domain $D(s)$.


Suppose that $c_i$ has endpoints in $D(s)$ and $D(t)$, for some $s \neq t$. Then, for any generator $u\notin\{s^{-1},t^{-1}\}$, we have that $\rho(u)(c_i)\in D(u)$.  
In addition, we have that
$\rho(s^{-1})(c_i)$ and $\rho(t^{-1})(c_i)$ belong to $\{c_0,\ldots,c_r\}$: indeed, the intersection $c_i\cap D(s)$ is nonempty; its image by $\rho(s^{-1})$ is contained in $\T\setminus (\Lambda\cup D(s^{-1}))$ and is adjacent to $D(s^{-1})$ because of Convention~\ref{reduced_con}; moreover, it has to intersect some $c_j$ because of Convention~\ref{Closures}. Then we must have $\rho(s^{-1})(c_i)=c_j$. The same holds for $t^{-1}$.
This implies that $c_i$ and $c_j$ are in the same orbit if and only if they are equivalent under the equivalence relation $\sim$ on 
$\{c_0,\ldots,c_r\}$ generated by 
\[c_i \sim c_j \quad\text{if there exists } t \in S \cup S^{-1} \text{ such that } c_i = \rho(t)(c_j) \text{ and } c_i \cap D(t) \neq \emptyset \]  
We will now argue that the number of equivalence classes under this relation can be 1 only if $n$ is even.  This is done by using the combinatorial data of the graphs from Definition~\ref{d:graph} to build a surface with boundary using the disc, and then making an Euler characteristic argument.   

For each generator $a\in S$, let $k(a)$ be the integer given by Proposition~\ref{p:graph}.
Let $P_a$ be a $4k(a)$-gon (topologically a disc) with cyclically ordered vertices $v_1, v_2, \ldots, v_{4k(a)}$.   Choose a connected component $I = [x_1,y_1]$ of $\overline{D(a)}$ and glue the oriented edge $v_1v_2$ to $I$ so as to agree with the orientation of $I \subset \T$.   Then glue the edge $v_3 v_4$ to the connected component of $\overline{D(a^{-1})}$ containing $\rho(a^{-1})(x_1)$, according to the orientation in $\T$.  Let $y_2$ denote the other endpoint of this connected component, and glue $v_5v_6$ to the connected component of $\overline{D(a)}$ containing $\rho(a)(y_2)$.  Iterate this process until all edges $v_{2j-1}v_{2j}$ have been glued to $\T = \partial \D$.  Our convention to follow the orientation of $\T$ implies that the resulting surface with boundary is orientable.  
Note that the remaining (unglued) edges of $P_a$ correspond exactly to the edges of the graph $\mathsf{\Gamma}_a$ from Definition~\ref{d:graph}; precisely, collapsing each connected component of $D(a)$ and of $D(a^{-1})$ to a point representing a vertex recovers the cycle $\mathsf{\Gamma}_a$.  

Now repeat this procedure for each generator in $S$, to obtain an orientable surface with boundary, which we will denote by $\Sigma$.  
A cartoon of the result of this procedure for the ping-pong action of Example~\ref{ex:strange} is shown in Figure~\ref{fig;surface_exotic}, and may be helpful to the reader.  

We claim that the number of boundary components of the surface $\Sigma$ is exactly the number of equivalence classes of the relation $\sim$.    To see this, we proceed as follows.  By construction, the connected components of $\partial \Sigma \cap \{c_0, \ldots, c_r\}$ are exactly the intervals $c_i$.  If some interval $c_i$ has endpoints in $D(s)$ and $D(t)$, then $\partial \Sigma \cap c_i$ is joined to $\rho(t^{-1}) c_i \cap \partial \Sigma$ and $\rho(s^{-1})c_i \cap \partial \Sigma$ by edges of $P_s$ and $P_t$ respectively.   Thus, $c_i \sim c_j$ implies that $c_i$ and $c_j$ lie in the same boundary component of $\Sigma$, and the intersection of that boundary component with $\{c_0, \ldots, c_r\}$ defines an equivalence class.  This proves the claim.  

We now compute the Euler characteristic of $\Sigma$ and conclude the proof.  
Proposition~\ref{p:graph} implies that the gluing of $P_a$ described in our procedure adds one face and $2k(a)$ edges to the existing surface. Therefore after all the polygons $P_a$ (as $a$ ranges over elements of $S$) have been glued, the surface $\Sigma$ obtained has  $\chi(\Sigma)\equiv n+1$ mod 2.   Since $\Sigma$ is orientable, $\chi(\Sigma)$ agrees mod 2 with the number of boundary components of $\Sigma$, which, by our claim proved above, agrees with the number of equivalence classes of $c_i$.  As discussed above, if $\rho$ is the dynamical realization of an isolated order then this number is equal to 1, hence $n+1 \equiv 1$ mod 2, and $n$ must be even.   
\end{proof}


The proof above can be improved to give a statement about general ping-pong actions:
\begin{thm}\label{t:gen_ping-pong}
Let $G=F_n$ be the free group of rank $n$ with free generating set $S$. Consider a ping-pong action $\rho$ of $(G,S)$ satisfying Conventions~\ref{reduced_con} and \ref{Closures}. Let $\Lambda(\rho)$  be the minimal invariant Cantor set for the action. Then the number of orbits of connected components of the complement $\T\setminus \Lambda(\rho)$ is congruent to $n+1$ mod 2.
\end{thm}

\begin{proof}
As in the previous proof, 
let $c_0, c_1,\ldots,c_r$ be the connected components of $\T \setminus \Lambda(\rho)$ that are not contained in any domain $D(s)$, and recall that $G$ permutes the connected components of $\T \setminus \Lambda(\rho)$.  We claim that each cycle of this permutation contains at least one of the $c_i$.   Given this claim, we may construct an orientable surface $\Sigma$ as in the proof of Theorem \ref{CO_thm}, whose boundary components count the number of cycles.  Computing Euler characteristic as above shows that the number of cycles is congruent to $n+1$ mod 2.  

We now prove the claim.  Suppose that $I$ is a connected component of $\T \setminus \Lambda(\rho)$ contained in some $D(s)$.  By Lemma \ref{GoodRealization}, we can take $\rho(s)$ to be piecewise linear, and such that each $\rho(s^{-1})$ expands $D(s)$ uniformly, increasing the length of each connected component by a factor of some $\mu > 1$, independent of $s$.  
Iteratively, assuming that $\rho(s_k s_{k-1} \cdots s_1) (I) \subset D(s_{k+1}^{-1})$, then the length of $\rho(s_{k+1} s_k\cdots s_1) (I)$ is at least $\mu^{k+1}\len(I)$.  This process cannot continue indefinitely, so some image of $I$ is not contained in a ping-pong domain.  
\end{proof}


\subsection{Exotic examples}

To indicate some of the potential difficulty of the problem of classifying all isolated orders on $F_n$, we give an example of a ping-pong configuration for $F_2$ that, even after applying an automorphism of $F_2$, cannot arise from the construction in Example  \ref{easy_ex}.  

\begin{figure}
\[
\includegraphics[scale=1]{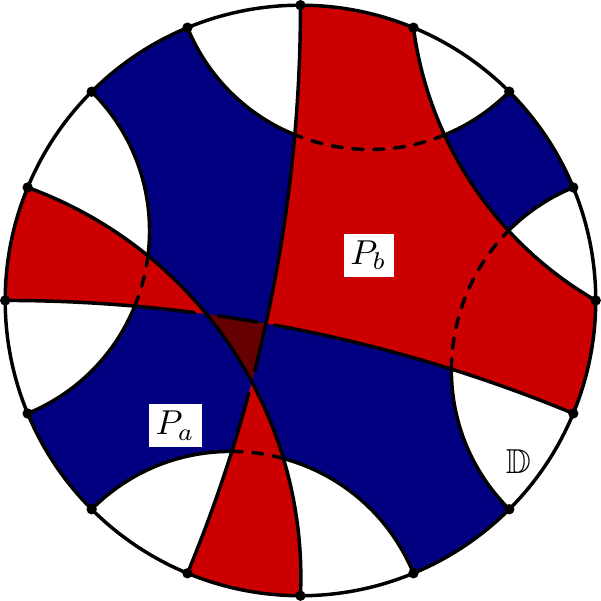}
\qquad
\includegraphics[scale=1]{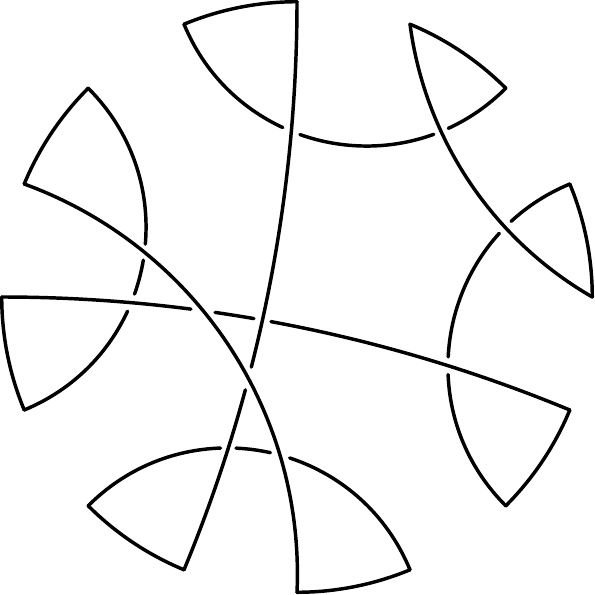}
\]
\caption{The surface associated with the exotic example (left), and its boundary component (right).}\label{fig;surface_exotic}
\end{figure}

\begin{figure}
\[
\includegraphics[scale=.9]{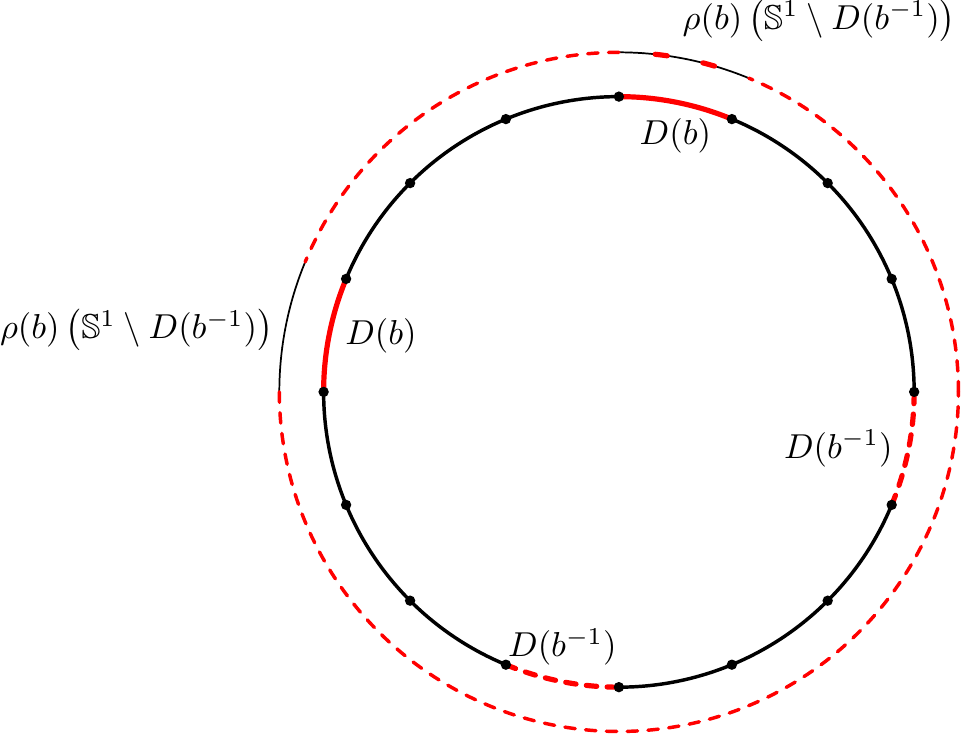}
\quad
\includegraphics[scale=.9]{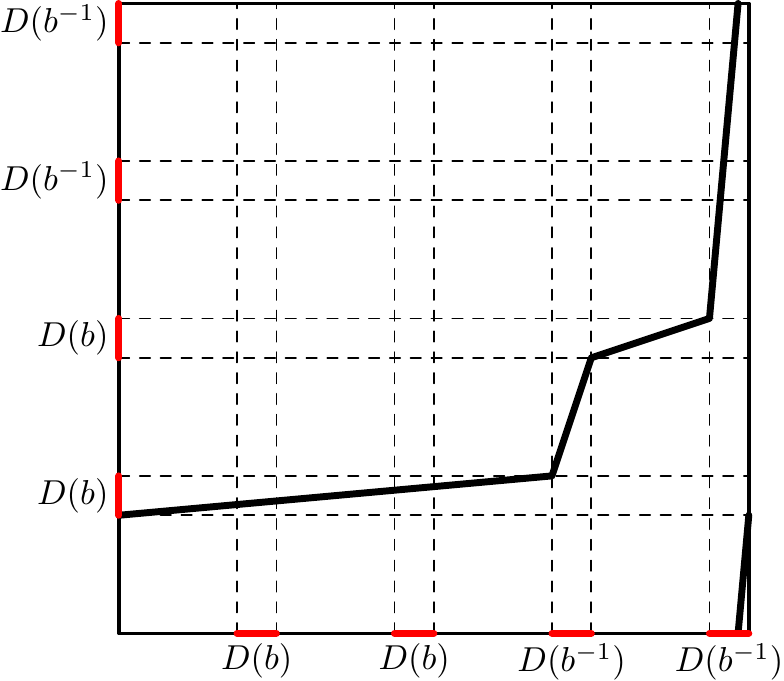}
\]
\caption{The ping-pong domains for $\rho(b)$ (left) and its graph (right). The circle is oriented counterclockwise.}\label{fig:strange}
\end{figure}

\begin{ex}\label{ex:strange}
Let $F_2 = \langle a, b\rangle$ and consider a ping-pong action where $\rho(b)$ is as defined by the graph in Figure~\ref{fig:strange}, and $\rho(a)$ is a hyperbolic element of $\SL(2,\R)$ chosen so that the connected components of the domains for the ping-pong action are in cyclic order as follows: 
\[ D(b^{-1}), D(a^{-1}), D(b^{-1}), D(a), D(b), D(a^{-1}), D(b), D(a) \]
(we are abusing notation slightly here, using each appearance of $D(s)$ to stand for a connected component of $D(s)$).
See Figure~\ref{fig;surface_exotic} (left) for an illustration of the domains and the surface $\Sigma$ constructed in the proof of Theorem \ref{CO_thm}.  

Since $\rho(b)$ has two hyperbolic fixed points, and $\rho(a)$ has four, this example is not realized by a ping-pong action in $\PSL(2,\R)$, nor in any finite extension of it.   In fact, the ping-pong configuration for $\rho(b)$ alone is atypical, in the sense that it is not the classical ping-pong configuration for a hyperbolic element in $\PSL(2,\R)$ -- $\rho(b)$ has a ``slow'' contraction on the left half of the circle, as two iterations of $\rho(b)$ are needed in order to bring the external gaps of $D(b)\cup D(b^{-1})$ into the component of $D(b)$ with the attracting fixed point.  

However, the surface $\Sigma$ from this construction has one boundary component, as shown in Figure~\ref{fig;surface_exotic} (right), so it corresponds to an isolated circular order in $\CO(F_2)$.
\end{ex}

Observe that one can create several examples of this kind, by choosing $\rho(b)$ to have two hyperbolic fixed points, but with an arbitrarily slow contraction (i.e.~with $N$ connected components for $D(b)$, $N\in\N$ arbitrary) and then choosing $\rho(a)$ to be a $N$-fold lift of a hyperbolic element in $\PSL(2,\R)$.    


\section{Left-invariant linear orders on $F_n\times \Z$}  \label{LO_sec}

The purpose of this section is to prove Theorem~\ref{LO_thm}, stating that $F_n\times \Z$ admits an isolated linear order if and only if $n$ is even.

\subsection{Preliminaries on linear orders}
\label{s:LOpre}

Before proceeding to the proof of Theorem \ref{LO_thm}, we recall some standard tools. 
As for circular orders, linear orders on countable groups have a \emph{dynamical realization} (see for instance \cite[Prop.~1.1.8]{GOD}). One quick way of seeing this given what we have already described, is by thinking of a linear order as a special case of a circular order.  
Indeed, given a linear order $\preceq$ on a group $G$, one defines the cocycle $c_\preceq$ by setting, for distinct $g_1,g_2,g_3\in G$
\[c_\preceq(g_1, g_2, g_3) = \mathsf{sign}(\sigma),\]
where $\sigma$ is the permutation of the indices such that $g_{\sigma(1)} \prec g_{\sigma(2)} \prec g_{\sigma(3)}$. 
Thus, the construction of the dynamical realization sketched in the proof of Proposition~\ref{p:realization} may be performed also for a linear order.  The result is an action on the circle with a single one global fixed point, which one can view as an action on the line with no global fixed point.   Conversely, a faithful action on the real line $\rho:G\to\Homeo_+(\R)$ can be viewed as a faithful action on the circle with a single fixed point, and the circular orders produced as in Remark~\ref{r:induce} will be linear orders on~$G$.

Next, we recall the notion of convex subgroups, their dynamical interpretation, and their relationship to isolated orders.

\begin{dfn}
 A subgroup $C$ in a linearly-ordered group $(G,\preceq)$ is \emph{convex} if for any two elements $h,k\in C$, and for any $g\in G$, the condition
$h\preceq g \preceq k$ implies $g\in C$. 
\end{dfn}

\begin{lem}[see \cite{GOD}, Prop.~2.1.3]\label{r:convex}
Let $G$ be a countable left-ordered group and consider a dynamical realization $\rho$ of $(G,\preceq)$ with basepoint $x$ such that $C$ is a convex subgroup.   Let $I$ be the interval bounded by $\inf_{h\in C}\rho(h)(x)$ and $\sup_{h\in C} \rho(h)(x)$.  Then $I$
has the following property:
\begin{equation}\label{interval}
\text{for any $g\in G$, either $\rho(g)(I) = I$, or $\rho(g)(I) \cap I = \emptyset$.}
\end{equation}
Moreover, the stabilizer of $I$ is precisely $C$. 

Conversely, given a faithful action on the real line $\rho:G\to\Homeo_+(\R)$, if an interval $I$ has the property \eqref{interval}, then the stabilizer $C=\stab{G}{I}$ is convex in any induced order with basepoint $x\in I$.
\end{lem}

It is easy to see that the family of convex subgroups of a linearly ordered group $(G,\preceq)$ forms a \emph{chain}: if $C_1,C_2$ are two convex subgroups of $(G,\preceq)$, then either $C_1\subset C_2$ or $C_2\subset C_1$. Moreover, for any convex subgroup $C\subset G$, the group $G$ acts on the ordered coset space $(G/C,\preceq_C)$ by order-preserving transformation.  (The induced order on the coset space is given by $fC<_C gC$ if and only if $fc< gc'$ for every $c,c'\in C$, which makes sense because $C$ is convex.)  In particular, this implies that if $C$ is convex in $(G,\preceq)$, then \emph{any} linear order $\preceq_C$ on $C$ may be extended to a (new) order $\preceq'$ on $G$ by declaring 
\[
id\preceq' g \Leftrightarrow \begin{cases}
C\preceq_C gC & \text{if }g\notin C, \\
id\preceq' g& \text{if }g\in C.
\end{cases}
\]
Elaborating on this, one can show the following lemma (see \cite[Prop. 3.2.53]{GOD} or \cite[Thm.~2]{MatsumotoLO} for details).
\begin{lem}\label{l:chain}
If $(G,\preceq)$ has an infinite chain of convex subgroups, then $\preceq$ is non-isolated in $\LO(G)$.
\end{lem}

Let us also introduce a dynamical property that implies that an order is non-isolated. 
Recall that two representations $\rho_1$ and $\rho_2: G \to \Homeo_+(\R)$ are \emph{semi-conjugate} if there is a proper, non-decreasing map $f: \R \to \R$ such that $f \circ \rho_1 = \rho_2 \circ f$.  

\begin{dfn}
Let $G$ be a discrete group.  Let $\mathsf{Rep}(G,\Homeo_+(\R))$ denote the space of representations (homomorphisms) $G \to \Homeo_+(\R)$, endowed with the compact-open topology; let $\mathsf{Rep}_\#(G,\Homeo_+(\R))$ be the subspace of representations with no global fixed points.

A representation $\rho \in \mathsf{Rep}_\#(G,\Homeo_+(\R))$ is said to be \emph{flexible} if every open neighborhood of $\rho$ in $ \mathsf{Rep}_\#(G,\Homeo_+(\R))$ contains a representation that is not semi-conjugate to $\rho$.
\end{dfn}

The following lemma is implicit in work of Navas \cite{NavasFourier} as well as in \cite{RivasLO}.  An explicit proof can be found in \cite[Prop.~2.8]{ABR}. 

\begin{lem}\label{r:Flexible}
Let $G$ be a discrete, countable group and let $\rho_0$ be the dynamical realization of an order~$\preceq$ with basepoint $x\in \R$. If $\rho_0$ is flexible, then $\preceq$ is non-isolated in $\LO(G)$.
\end{lem}

\begin{rem}
Though not needed for our work here, we note that a precise characterization of isolated circular and linear orders in terms of a strong form of rigidity (i.e.~strong non-flexibility) of their dynamical realizations is given in \cite{MR}.
\end{rem}

As mentioned in the introduction, in order to prove Theorem~\ref{LO_thm} we use the relationship between circular orders on groups and linear orders on their central extensions by $\Z$.  For this purpose, we need the notion of {\em cofinal elements}.

\begin{dfn}
An element $h$ in a linearly-ordered group $(G, \preceq)$ is called \emph{cofinal} if
\begin{equation}\label{eq:cofinal}
\text{for all $g\in G$,
there exist $m,n\in\Z$ such that $h^m \preceq g \preceq h^n$.}
\end{equation}
\end{dfn}

\begin{rem}\label{r:cofinal}
Cofinal elements also have a characterization in terms of the dynamical realization:  If $\rho$ is a dynamical realization of $\preceq$ with basepoint $x$, then $h \in G$ is cofinal if and only if $\rho(h)$ has no fixed point.   Indeed, if $h$ is not cofinal, then the point $\inf\{\rho(g)(x)\mid h^n\preceq g\text{ for every }n\in \Z\}$ is fixed by $\rho(h)$.  Conversely, if $h$ satisfies \eqref{eq:cofinal}, then the orbit of $x$ under $\rho(h)$ is clearly unbounded on both sides.  
\end{rem}

Given a group $G$ with a circular order $c$, there is a natural procedure to \emph{lift $c$ to} a linear order~$\preceq_c$ on a central extension of $G$ by $\Z$ such that any generator of the central $\Z$ subgroup is cofinal for $\preceq_c$ \cite{zheleva,MR}. 
The following statement appears as Proposition 5.4 in~\cite{MR}.
\begin{prop}\label{LO_central}
Assume that $G$ is finitely generated and $c$ is an isolated circular order on $G$. If $\preceq_c$  is the lift of $c$ to a central extension  $\widehat{G}$ of $G$ by $\Z$, then the induced linear order $\preceq_c$ is isolated in $\LO(\widehat{G})$.
\end{prop}
\smallskip

Finally, we recall the definition of crossings. 

\begin{dfn}
Let $G$ be a group acting on a totally ordered space $(\Omega,\leq)$. The action has \emph{crossings} if there exist $f,g\in G$ and $u,v,w\in \Omega$ such that:
\begin{enumerate}
\item $u < w < v$.
\item $g^n u < v$ and $f^n v > u$ for every $n \in \N$, and
\item there exist $M, N$ in $\N$ such that $f^N v < w < g^M u$.
\end{enumerate}
In this case, we say that $f$ and $g$ are \emph{crossed}.
\end{dfn}

If $f$ and $g$ are crossed, then the graph of $\rho(f)$ and $\rho(g)$ in the dynamical realization is locally  given by the picture in Figure~\ref{fig:crossing}.

\begin{figure}
\[
\includegraphics[scale=.8]{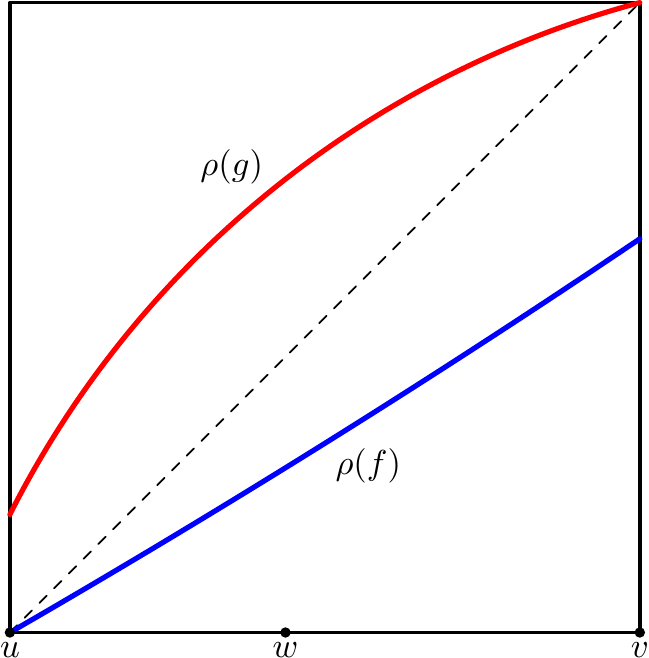}
\]
\caption{A crossing in the dynamical realization.}\label{fig:crossing}
\end{figure}

Our application of the notion of crossings will be through the following lemma. 
\begin{lem}[\cite{GOD} Cor.~3.2.28]
\label{l:crossings}
Let $C$ be a convex subgroup of $(G,\preceq)$ and suppose that the (natural) action of $G$ on $(G/C,\preceq_C)$ has no crossings. 
Then there exists a homomorphism $\tau:G\to \R$ with $C$ in its kernel. Moreover, if $C$ is the maximal convex subgroup of $(G,\preceq)$, then $C$ agrees with the kernel of $\tau$.
\end{lem}

\subsection{Isolated linear orders on $F_n\times \Z$}  \label{sec_isoFnZ}

We now turn to our main goal of describing isolated linear orders on $F_n\times \Z$ and proving Theorem~\ref{LO_thm}.  We begin by reducing the proof to the statement of Proposition \ref{prop flexible} below.  

Since every central extension of $F_n$ by $\Z$ splits, Proposition~\ref{LO_central} tell us that $F_{2n}\times\Z$ admits isolated linear orders -- more precisely, any lift of an isolated order on $F_{2n}$ to $F_{2n}\times\Z$ will be isolated.  
Furthermore, if 
\[
0\to\Z  \to \hat{G} \overset{\pi}\rightarrow G\to 1
\]
is a central extension of $G$ by $\Z$ then any linear order $\preceq$ on $\hat{G}$ in which $\Z$ is cofinal gives a canonical circular order on $G$ as follows. Let $z$ be the generator of $\Z$ such that $z \succ id$.  Since $z$ is cofinal, for each $g \in G$, there exists a unique representative $\hat{g} \in \pi^{-1}(g)$ such that $id \preceq \hat{g} \prec z$. Given distinct elements $g_1, g_2, g_3 \in G$, let $\sigma$ be the permutation such that 
\[id \preceq \hat{g}_{\sigma(1)} \prec \hat{g}_{\sigma(2)} \prec \hat{g}_{\sigma(3)} \prec z.\]
   Define $\pi^*(\preceq)(g_1, g_2, g_3) := \mathsf{sign}(\sigma)$.  One checks that this is a well defined circular order on $G$.  In the proof of Proposition 5.4 of \cite{MR}, it is shown that $\pi^*$ is continuous and is locally injective when $G$ is finitely generated, which implies that an isolated linear order of $F_{2n+1}\times \Z$ with cofinal center induces an isolated circular order of $F_{2n+1}$ by this procedure. Since $F_{2n+1}$ has no isolated circular orders, by Theorem \ref{CO_thm}, to finish the proof of Theorem \ref{LO_thm} it is enough to show the following:

\begin{prop}\label{prop flexible} Let $F$ be a free group, and $\preceq$ a linear order on $G=F\times \Z$ in which the central factor is not cofinal. Then $\preceq$ is non-isolated.
\end{prop}

As a warm-up, as well as tool to be used in the proof, we start with a short proof of a special case.  
\begin{lem}\label{l:infinite_rank}
Let $F$ be a free group of infinite rank and $G=F\times \Z$. Then no order in $\LO(G)$ is isolated.
\end{lem}

\begin{proof}
Let $f_1,f_2,\ldots$ be a set of free generators of 
the free factor $F$ and $g$ the generator of the central factor $\Z$.
Let $\preceq$ be any order on $G$ and $\rho_0$ a dynamical realization with basepoint $x$.
For any fixed $n\in \N$, we can define a representation $\rho_n:G\to \Homeo_+(\R)$ by setting
\[\rho_n(g)=\rho_0(g),\quad
\rho_n(f_k)=\begin{cases}
\rho_0(f_k) &\text{if }k\neq n,\\
\rho_0(f_n)^{-1}&\text{if }k=n.
\end{cases}
\] 
It is easy to see that the orbit of $x$ is free for all the actions $\rho_n$, and that no two distinct representations $\rho_n$ and $\rho_m$ are semi-conjugate one to another.  Thus, they determine distinct orders $\preceq_n$; and these orders converge to $\preceq$ in $\LO(G)$ as $n\to\infty$.
\end{proof}

\setcounter{claim}{0}

\begin{proof}[Proof of Proposition \ref{prop flexible}]
We have already eliminated the case where $F$ has infinite rank.  If $F$ has rank one, then $F \times \Z$ is abelian, and 
so admits no isolated orders (see~\cite{sikora}).  
So from now on we assume that the rank of $F$ is finite and at least $2$.  

Looking for a contradiction, suppose that $\preceq$ is a linear order on $G = F\times \Z$ which is isolated, and in which the center is not cofinal.  Let $\rho$ be its dynamical realization, and let $z$ be a generator of the central $\Z$ subgroup.  By Remark \ref{r:cofinal}, $\rho(z)$ acts with fixed points.  Moreover since $\Z$ is central, the set of fixed points of $\rho(z)$ is $\rho(G)$-invariant.  Since $\rho(G)$ has no global fixed point, this implies that $\rho(z)$ has fixed points in every neighborhood of $+\infty$ and of $-\infty$.  

We now find a convex subgroup in which $z$ is cofinal.   Let $I$ denote the connected component of $\R\setminus \Fix(\rho(z))$ that contains the basepoint $x_0$ of~$\rho$, so in particular $I$ is a bounded interval.  Let $C=\stab{G}{I}$. 

\begin{claim}
$C$ is a convex subgroup of $(G,\preceq)$ and $z$ is cofinal in $(C,\preceq)$.
\end{claim}

\begin{proof}[Proof of Claim]
If $h,k\in C$ and $g\in G$ satisfy  $\rho(h)(x_0)<\rho(g)(x_0)<\rho(k)(x_0)$, 
then for any $n$ we also have $\rho(z^nh)(x_0)<\rho(z^ng)(x_0)<\rho(z^nk)(x_0)$.  Since $z$ is central, this implies 
\begin{equation}\label{eq:convex}
\rho(hz^n)(x_0)<\rho(gz^n)(x_0)<\rho(kz^n)(x_0).
\end{equation}
Up to replacing $z$ with $z^{-1}$, without loss of generality we may assume that $\rho(z)(x_0)>x_0$. Thus, as $n\to\infty$, the sequence of points $\rho(z^{n})(x_0)$ converges to the rightmost point of $I$, which is fixed by both $\rho(h)$ and $\rho(k)$. We deduce from  \eqref{eq:convex} that $\rho(g)$ also fixes this point. Similarly, considering the limit in \eqref{eq:convex} as $n\to-\infty$ shows that the leftmost point of $I$ is fixed by $\rho(g)$. Hence $g\in C$, which shows $C$ is convex.  
Finally, by Remark~\ref{r:cofinal}, the fact that $\rho(z)$ has no fixed points in $I$ implies that $z$ is cofinal in~$(C,\preceq)$.
\end{proof}

Since $(G, \preceq)$ is isolated and $C$ is convex, the restriction of $\preceq$ to $C$ is also an isolated order on $C$.  
Additionally, the fact that $\preceq$ is isolated implies, by Lemma~\ref{l:chain}, that the chain of convex subgroups of $G$ is \emph{finite}.   Let $G'$ denote the smallest convex subgroup properly containing $C$.   Since $\Z \subset C$, we have that $G'$ is also a direct product of $\Z$ and a free group (a subgroup of $F$), and again, our assumptions on $\preceq$ imply that (the restriction of) $\preceq$ is an isolated left-order on $G'$ in which $\Z$ is not cofinal.    Thus, we may work from now on with $G'$ instead of $G$.  Equivalently -- and, for notational convenience, this is how we will proceed -- we may assume that $C$ is the \emph{maximal} convex subgroup of $G$.  

For our next claim observe that this maximal convex subgroup $C$ also admits a decomposition of the form $F^*\times\Z$, where $F^*$ is a subgroup of $F$. 

\begin{claim}\label{cl:F*}
$F^*$ is a non trivial free group of even rank.
\end{claim}

\begin{proof}[Proof of Claim]  Since the restriction of $\preceq$ to $C=F^*\times \Z$ is isolated, as before, Lemma~\ref{l:infinite_rank} implies that $F^*$ cannot have infinite rank.   If $F^*$ were trivial, then the action of $G$ would be semi-conjugate to an action of $F$, thus making very easy to perturb the action of $F\times \Z$ and thus the order $\preceq$ (recall that free groups have no isolated orders \cite{McCleary2}).  Thus, $F^*$ is a nontrivial free group of finite rank, and as $z$ is cofinal in the ordering in $C$, its rank must be even (c.f. the remarks at the beginning of Section \ref{sec_isoFnZ}).  
\end{proof}

\begin{claim}\label{InfiniteIndex}
$C$ has infinite index in $G$.
\end{claim} 

\begin{proof}[Proof of Claim]
If $C$ had finite index, then the $G$-orbit of the interval $I$ would be bounded. This would imply that the dynamical realization has a global fixed point, which is absurd.
\end{proof}

Since every nontrivial normal, infinite index subgroup of $F$ has infinite rank, we conclude from Claims~\ref{cl:F*} and~\ref{InfiniteIndex} that $F^*$ (and thus $C$) is not a normal subgroup of $G$.  Lemma~\ref{l:crossings} thus implies that the action of $G$ on $(G/C,\preceq_C)$ has crossings, as otherwise $C$ would be normal.
In particular, if we collapse $I$ and its $G$-orbit, we obtain a semi-conjugate action $\bar \rho: G\to \Homeo_+(\R)$ which is minimal and has crossings.    Using this observation, we now prove the following claim.   
 
\begin{claim}\label{cl:flex}
For any compact set $K \subset \R$, there exists $\rho'$ agreeing with $\rho$ on $K$, but not semi-conjugate to~$\rho$. 
\end{claim}

\begin{proof}[Proof of Claim]
Fix a compact set $K$. We will modify the action of $F$ outside $K$ to produce an action of $G$ that is not semi-conjugate to $ \rho$.

Suppose as an initial case that there is a primitive element  (i.e.~a generator in some free generating set) $a$ of $F$ such that $\rho(a)$ has a fixed point $p \notin K$.  Without loss of generality, assume $p$ is to the right of $K$, the other case is completely analogous.   
Since $\Fix(\rho(a))$ is $\rho(z)$-invariant, and $\rho(z^n)(p)$ is bounded and accumulates at a fixed point of $\rho(z)$, we may also assume without loss of generality that we have chosen $p$ to be a common fixed point of $\rho(z)$ and $\rho(a)$.  

We now define $a_+$ and $a_- \in \Homeo_+(\R)$ which commute with $\rho(z)$, and have the property that $a_+(x) \geq x$ for all $x \geq p$, and $a_- \circ \rho(a)(x) \leq x$ for all $x \geq p$.
For this, let $J$ be any connected component of $(\R \setminus \Fix(\rho(z))) \cap [p, \infty)$.
Suppose first that $J$ contains a point of $\Fix(\rho(a))$. The fact that $a$ and $z$ commute means that the endpoints of $J$ are preserved by $\rho(a)$.  Then define the restriction of $a_+$ to $J$ to agree with $\rho(z)$ if $\rho(z)(x) > x$ on $J$, or with $\rho(z^{-1})$ if $\rho(z)(x) < x$ on $J$.
If $J$ contains no point of $\Fix(\rho(a))$, then $J \subseteq J'$ where $J'$ is a connected component of $\R \setminus \Fix(\rho(a)) \cap [p, \infty)$, and we may define $a_+$ to agree with $\rho(a)$ or $\rho(a^{-1})$ there, so as to satisfy $a_+(x)>x$ for $x\in J$.  
Lastly, set $a_+(x)=x$ for any $x\in \Fix(\rho(z))\cap [p,\infty)$.
 The definition of $a_-$ is analogous.  
Let $\rho_\pm$ be the actions obtained by replacing the action of $\rho(a)$ by that of $a_\pm$ on $[p, \infty)$ and leaving the other generators unchanged.  Since $a_+$ and $a_-$ commute with $\rho(z)$, this defines representations of $G$, and clearly $\rho_+$ and $\rho_-$ are not semi-conjugate. 


We are left to deal with the case where no primitive element of $F$ has a fixed point outside $K$.  In this case, we will perturb the action $\rho$ to obtain a primitive element with a fixed point outside $K$, and hence a non-semi-conjugate action.  To do this, we use the fact that the semi-conjugate action $\bar \rho$ has crossings and is minimal.  Minimality implies that crossings can be found outside any compact set and, thus, for any compact $K\subset \R$ there is $g\in G$ such that $\R\setminus \Fix(\rho(g))$ has a component outside (and on the right of) $K$. Let $J=(j_0,j_1)$ denote one of those components. 

Notice that some primitive element $a \in F$ has the property that $\rho(a)(j_0)\in J$, but $\rho(a)(j_1)\notin J$.  For if this was not the case, $J$ would satisfy property \eqref{interval} and as observed in Remark~\ref{r:convex}, $G$ would then have a convex subgroup properly containing a conjugate of $C$.  Since $C$ was assumed maximal, this is impossible.      

Fix a primitive element $a$ with the property above, and let $\bar g$ be the homeomorphism defined as the identity outside $J$ and agreeing with $\rho(g)$ on $J$. Define $\rho^{\bar g}$ by 
\[\rho^{\bar g}(a)= \bar{g}\rho(a), \text{ and } \rho^{\bar g}(b)=\rho(b) \text{  for any other generator of $F$, and } \rho^{\bar g}(z) = \rho(z).\]

Since $\bar g$ commutes with $\rho(z)=\rho^{\bar g}(z)$, the new action $\rho^{\bar g}$ is a representation of $G$. Moreover, by changing $\bar g$ by some power if necessary, we have that $\rho^{\bar g}(a)$ has a fixed point in $J$. This ends the proof of Claim~\ref{cl:flex}.\end{proof}

To finish the proof of Proposition \ref{prop flexible} (and thus that of Theorem \ref{LO_thm}), we note that the \emph{flexibility} of $\rho$ from Claim~\ref{cl:flex} together with the statement of Lemma~\ref{r:Flexible} implies that the order is non-isolated, giving the desired contradiction.
\end{proof}

\section*{Acknowledgments}

The authors thank Yago Antol\'in for suggesting Corollary~\ref{cor:marked}.
K.M.~was partially supported by NSF grant DMS-1606254. C.R.~was partially supported by FONDECYT 1150691. M.T.~was partially supported by PEPS -- Jeunes Chercheur-e-s -- 2017 (CNRS), Projet ``Jeunes G\'eom\'etres'' of F.~Labourie (financed by the Louis D.~Foundation) and the R\'eseau France-Br\'esil en Math\'ematiques.

\bibliographystyle{plainurl}

\begin{bibdiv}
\begin{biblist}

\end{biblist}
\end{bibdiv}


\begin{bibdiv}
\begin{biblist}

\bib{ABR}{article}{
      author={Alonso, J.},
      author={Brum, J.},
      author={Rivas, C.},
       title={Orderings and flexibility of some subgroups of
  {$Homeo_+(\mathbb{R})$}},
        date={2017},
        ISSN={0024-6107},
     journal={J. Lond. Math. Soc. (2)},
      volume={95},
      number={3},
       pages={919\ndash 941},
      review={\MR{3664524}},
}

\bib{markov-partition-vfree}{article}{
      author={Alvarez, S.},
      author={Barrientos, P.},
      author={Filimonov, D.},
      author={Kleptsyn, V.},
      author={Malicet, D.},
      author={Meni\~no, C.},
      author={Triestino, M.},
       title={Maskit partitions and locally discrete groups of real-analytic
  circle diffeomorphisms},
        note={In preparation},
}

\bib{Arora-McCleary}{article}{
      author={Arora, A.K.},
      author={McCleary, S.H.},
       title={Centralizers in free lattice-ordered groups},
        date={1986},
        ISSN={0362-1588},
     journal={Houston J. Math.},
      volume={12},
      number={4},
       pages={455\ndash 482},
      review={\MR{873641}},
}

\bib{BS}{article}{
      author={{Baik}, H.},
      author={{Samperton}, E.},
       title={{Spaces of invariant circular orders of groups}},
     journal={Groups Geom. Dyn.},
        note={To appear},
}

\bib{Calegari}{inproceedings}{
      author={Calegari, D.},
       title={Circular groups, planar groups, and the {E}uler class},
        date={2004},
   booktitle={Proceedings of the {C}asson {F}est},
      series={Geom. Topol. Monogr.},
      volume={7},
   publisher={Geom. Topol. Publ., Coventry},
       pages={431\ndash 491},
         url={http://dx.doi.org/10.2140/gtm.2004.7.431},
      review={\MR{2172491}},
}

\bib{CC1}{article}{
      author={Cantwell, J.},
      author={Conlon, L.},
       title={Foliations and subshifts},
        date={1988},
        ISSN={0040-8735},
     journal={Tohoku Math. J. (2)},
      volume={40},
      number={2},
       pages={165\ndash 187},
         url={http://dx.doi.org/10.2748/tmj/1178228024},
      review={\MR{943817}},
}

\bib{CC2}{article}{
      author={Cantwell, J.},
      author={Conlon, L.},
       title={Leaves of {M}arkov local minimal sets in foliations of
  codimension one},
        date={1989},
        ISSN={0214-1493},
     journal={Publ. Mat.},
      volume={33},
      number={3},
       pages={461\ndash 484},
         url={http://dx.doi.org/10.5565/PUBLMAT_33389_07},
      review={\MR{1038484}},
}

\bib{champ-guira}{article}{
      author={Champetier, C.},
      author={Guirardel, V.},
       title={Limit groups as limits of free groups},
        date={2005},
        ISSN={0021-2172},
     journal={Israel J. Math.},
      volume={146},
       pages={1\ndash 75},
         url={http://dx.doi.org/10.1007/BF02773526},
      review={\MR{2151593}},
}

\bib{CMR}{article}{
      author={{Clay}, A.},
      author={{Mann}, K.},
      author={{Rivas}, C.},
       title={{On the number of circular orders on a group}},
        date={2017},
     journal={ArXiv e-prints},
      eprint={1704.06242},
}

\bib{dehornoy14}{article}{
      author={Dehornoy, P.},
       title={Monoids of {$O$}-type, subword reversing, and ordered groups},
        date={2014},
        ISSN={1433-5883},
     journal={J. Group Theory},
      volume={17},
      number={3},
       pages={465\ndash 524},
         url={http://dx.doi.org/10.1515/jgt-2013-0049},
      review={\MR{3200370}},
}

\bib{DehornoyBook}{book}{
      author={Dehornoy, P.},
      author={Dynnikov, I.},
      author={Rolfsen, D.},
      author={Wiest, B.},
       title={Ordering braids},
      series={Mathematical Surveys and Monographs},
   publisher={American Mathematical Society, Providence, RI},
        date={2008},
      volume={148},
        ISBN={978-0-8218-4431-1},
         url={http://dx.doi.org/10.1090/surv/148},
      review={\MR{2463428}},
}

\bib{DKN2009}{article}{
      author={Deroin, B.},
      author={Kleptsyn, V.A.},
      author={Navas, A.},
       title={On the question of ergodicity for minimal group actions on the
  circle},
        date={2009},
     journal={Mosc. Math. J.},
      volume={9},
      number={2},
       pages={263\ndash 303},
      review={\MR{2568439}},
}

\bib{DKN2014}{article}{
      author={Deroin, B.},
      author={Kleptsyn, V.A.},
      author={Navas, A.},
       title={On the ergodic theory of free group actions by real-analytic
  circle diffeomorphisms},
        date={2014},
     journal={Invent. Math.},
      note = {To appear},
}

\bib{GOD}{book}{
      author={Deroin, B.},
      author={Navas, A.},
      author={Rivas, C.},
       title={Groups, orders and dynamics},
   publisher={Arxiv e-print (2016), available at  1408.5805},
        note={},
}

\bib{DD}{article}{
      author={Dubrovina, T.~V.},
      author={Dubrovin, N.~I.},
       title={On braid groups},
        date={2001},
        ISSN={0368-8666},
     journal={Mat. Sb.},
      volume={192},
      number={5},
       pages={53\ndash 64},
         url={http://dx.doi.org/10.1070/SM2001v192n05ABEH000564},
      review={\MR{1859702}},
}

\bib{FK2012}{article}{
      author={Filimonov, D.A.},
      author={Kleptsyn, V.A.},
       title={Structure of groups of circle diffeomorphisms with the property
  of fixing nonexpandable points},
        date={2012},
     journal={Funct. Anal. Appl.},
      volume={46},
      number={3},
       pages={191\ndash 209},
      review={\MR{3075039}},
}

\bib{Ghys}{article}{
      author={Ghys, \'E.},
       title={Groups acting on the circle},
        date={2001},
        ISSN={0013-8584},
     journal={Enseign. Math. (2)},
      volume={47},
      number={3-4},
       pages={329\ndash 407},
      review={\MR{1876932}},
}

\bib{IM}{article}{
      author={Inaba, T.},
      author={Matsumoto, S.},
       title={Resilient leaves in transversely projective foliations},
        date={1990},
        ISSN={0040-8980},
     journal={J. Fac. Sci. Univ. Tokyo Sect. IA Math.},
      volume={37},
      number={1},
       pages={89\ndash 101},
      review={\MR{1049020}},
}

\bib{ItoAlg}{article}{
      author={Ito, T.},
       title={Dehornoy-like left orderings and isolated left orderings},
        date={2013},
        ISSN={0021-8693},
     journal={J. Algebra},
      volume={374},
       pages={42\ndash 58},
         url={http://dx.doi.org/10.1016/j.jalgebra.2012.10.016},
      review={\MR{2998793}},
}

\bib{mane}{article}{
      author={Ma\~n\'e, R.},
       title={Hyperbolicity, sinks and measure in one-dimensional dynamics},
        date={1985},
        ISSN={0010-3616},
     journal={Comm. Math. Phys.},
      volume={100},
      number={4},
       pages={495\ndash 524},
         url={http://projecteuclid.org/euclid.cmp/1104114003},
      review={\MR{806250}},
}

\bib{MR}{article}{
      author={{Mann}, K.},
      author={{Rivas}, C.},
       title={{Group orderings, dynamics, and rigidity}},
        date={2016},
     journal={Ann. Inst. Fourier (Grenoble)},
      note = {To appear},
}

\bib{Matsumoto}{incollection}{
      author={Matsumoto, S.},
       title={Measure of exceptional minimal sets of codimension one
  foliations},
        date={1988},
   booktitle={A f\^ete of topology},
   publisher={Academic Press, Boston, MA},
       pages={81\ndash 94},
      review={\MR{928398}},
}

\bib{BP}{article}{
      author={Matsumoto, S.},
       title={Basic partitions and combinations of group actions on the circle:
  a new approach to a theorem of {K}athryn {M}ann},
        date={2016},
     journal={Enseign. Math.},
      volume={62},
      number={1-2},
       pages={15\ndash 47},
      review={\MR{3605808}},
}

\bib{MatsumotoLO}{article}{
      author={Matsumoto, S.},
       title={Dynamics of isolated left orders},
        date={2017},
     journal={ArXiv e-prints},
      eprint={1701.05987},
}

\bib{McCleary2}{article}{
      author={McCleary, S.H.},
       title={Free lattice-ordered groups represented as {$o$}-{$2$} transitive
  {$l$}-permutation groups},
        date={1985},
        ISSN={0002-9947},
     journal={Trans. Amer. Math. Soc.},
      volume={290},
      number={1},
       pages={69\ndash 79},
         url={http://dx.doi.org/10.2307/1999784},
      review={\MR{787955}},
}

\bib{NavasFourier}{article}{
      author={Navas, A.},
       title={On the dynamics of (left) orderable groups},
        date={2010},
        ISSN={0373-0956},
     journal={Ann. Inst. Fourier (Grenoble)},
      volume={60},
      number={5},
       pages={1685\ndash 1740},
         url={http://aif.cedram.org/item?id=AIF_2010__60_5_1685_0},
      review={\MR{2766228}},
}

\bib{NavasAlg}{article}{
      author={Navas, A.},
       title={A remarkable family of left-ordered groups: central extensions of
  {H}ecke groups},
        date={2011},
        ISSN={0021-8693},
     journal={J. Algebra},
      volume={328},
       pages={31\ndash 42},
         url={http://dx.doi.org/10.1016/j.jalgebra.2010.10.020},
      review={\MR{2745552}},
}

\bib{RivasLO}{article}{
      author={Rivas, C.},
       title={Left-orderings on free products of groups},
        date={2012},
        ISSN={0021-8693},
     journal={J. Algebra},
      volume={350},
       pages={318\ndash 329},
         url={http://dx.doi.org/10.1016/j.jalgebra.2011.10.036},
      review={\MR{2859890}},
}

\bib{sikora}{article}{
      author={Sikora, A.S.},
       title={Topology on the spaces of orderings of groups},
        date={2004},
        ISSN={0024-6093},
     journal={Bull. London Math. Soc.},
      volume={36},
      number={4},
       pages={519\ndash 526},
         url={http://dx.doi.org/10.1112/S0024609303003060},
      review={\MR{2069015}},
}

\bib{zheleva}{article}{
      author={\v{Z}eleva, S.D.},
       title={Cyclically ordered groups},
        date={1976},
        ISSN={0037-4474},
     journal={Sibirsk. Mat. \v{Z}.},
      volume={17},
      number={5},
       pages={1046\ndash 1051, 1197},
      review={\MR{0422106}},
}

\end{biblist}
\end{bibdiv}

\end{document}